\documentclass[a4paper]{amsproc}
\usepackage{amssymb}
\usepackage{listings}
\usepackage{tabularx}
\usepackage{kpfonts}
\usepackage{epstopdf}
\usepackage{amscd}
\usepackage{graphicx}

\theoremstyle{plain}
 \newtheorem{thm}{Theorem}[section]

\theoremstyle{definition}

\theoremstyle{remark}
 
 \numberwithin{equation}{section}

\renewcommand{\le}{\leqslant}
\renewcommand{\ge}{\geqslant}

\title[On the number of the equivalence classes of invertible Boolean functions]{On the number of equivalence classes of invertible Boolean functions under action of permutation of variables on domain and range}



\author[Cari\'c]{\bfseries Marko Cari\'c}

\address{
Advanced School of Electrical Engineering Applied Studies\\ 
Belgrade, Serbia \\
} \email{caric.marko@gmail.com}

\author[\v{Z}ivkovi\'c]{Miodrag \v{Z}ivkovi\'c}
\address{
Faculty of Mathematics \\
Department of Informatics \\
University of Belgrade, Serbia \\
 }
\email{ezivkovm@matf.bg.ac.rs}

\begin{document}

\vspace{18mm} \setcounter{page}{1} \thispagestyle{empty}

\begin{abstract}
Let $V_n$ be the number of  equivalence classes of invertible maps
from $\{0,1\}^n$ to $\{0,1\}^n$, under action of permutation of
variables on domain and range. So far, the values $V_n$ have been
known for $n\le 6$. This paper describes the procedure by which the
values of $V_n$ are calculated for $n\le 30$.
\end{abstract}

\maketitle

\section{Introduction}

Let $V_n$ be the number of  equivalence classes of invertible maps
from $\{0,1\}^n$ to $\{0,1\}^n$, under action of permutation of
variables on domain and range. Lorens~\cite{1} gave a method for
calculating the number of equivalence classes of invertible Boolean
functions under the following group operations on the input and
output variables: complementation, permutation, composition of
complementation and permutation, linear transformations and affine
transformations. In particular, he calculated the values $V_n$ for
$n\le 5$. Irvine~\cite{4} in 2011. calculated $V_6$ (the sequence
A000653). In this paper using a more efficient procedure, the values
$V_n$ are calculated for $n\le 30$.

\section{Notation}

Let $S_r$ denote symmetric group on $r$ letters. Consider a set of
vectorial invertible Boolean functions (hereinafter referred to as
functions), i.e. the set $S_N$ of permutations of $B_n=\{0,1\}^n$,
where $N=2^n$. The function $F\in S_N$ maps the $n$-tuple
$X=(x_1,...,x_n)\in B_n$ into $Y=(y_1,...,y_n)=F(X)$. For some
permutation $\sigma \in S_n$, the result of its action on
$X=(x_1,...,x_n)\in B_n$ is
${\sigma}'(X)=(x_{\sigma(1)},...,x_{\sigma(n)})\in B_n$.

An arbitrary pair $(\rho,\sigma)\in S_n^2$ determines mapping
$T_{\rho,\sigma}:\,S_N\rightarrow S_N$, defined by
$T_{\rho,\sigma}(F)={\rho}'\circ F \circ {\sigma}'$ where $F\in
S_N$; in other words, if $F'=T_{\rho, \sigma}(F) $ then $F'(X)=
\rho'(F(\sigma'(X)))$ for all $X\in B_n$. The set of all mappings
$T_{\rho, \sigma}$ with respect to composition is a subgroup of
$S_{N!}$.

The two functions $F,H\in S_N$ are considered equivalent if there
exist permutations $\rho, \sigma \in S_n$ such that
 $H=T_{\rho,\sigma}(F)$, i.e.
if they differ only by a permutation of input or output variables.

Let $\iota$ denote the identity permutation. Every permutation
$\sigma \in S_n$ uniquely determines the permutation ${\sigma}'\in
S_N$. Let $S_n'$ denote the subgroup of $S_N$ consisting of all
permutations ${\sigma}'$ corresponding to permutations $\sigma\in
S_n$. The mapping  $\sigma \mapsto \sigma'$ is a monomorphism from
$S_n$ to $S_N$ (see \cite{2}).

Let  $\sigma\in  S_r$. Let $p_i$, $1\le i\le r$, denote the number
of cycles of length $i$ in a cycle decomposition of $\sigma$; here
$\sum_{i=1}^r ip_i=r$. The cycle index monomial of $\sigma$ is the
product $\prod_{i=1}^r t_i^{p_i}$, where $t_i$, $1\le i\le r$, are
independent variables. It can be equivalently described by the
vector $spec(\sigma)=p=(p_1,p_2,\ldots,p_r)$. For an arbitrary
positive integer $n$ let $P_n=\{(p_1,p_2,\ldots,p_n) \mid p_i\ge
0,\sum_{i=1}^n i p_i =n\}$ denote the set of partitions of $n$. For
some $p\in P_n$ let $S_{n,p} = \{ \sigma\in S_n \mid
spec(\sigma)=p\}$. An arbitrary partition $p$ corresponds to the
decomposition $n=k_{p,1}+k_{p,2}+\cdots+k_{p,m(p)}$ into positive
summands $k_{p,1}\ge k_{p,2}\ge\cdots k_{p,m(p)}>0$ where summand
$i=n,n-1,\dots,1$ in this sum appears $p_i$ times.

Let $<r,s>$ and $(r,s)$ denote the least common multiple and the
greatest common divisor of $r$ and $s$, respectively.

\section{Preliminaries}

The calculation of $V_n$ is based on the following known facts (see
e.g. \cite{1,2,3}):
\begin{enumerate}
\item The cardinality of  $S_{n,p}$ equals to
$$
|S_{n,p}|=\frac{n!}{\prod_{i} i^{p_i} p_i!}.
$$
\item
Let $\sigma_1$, $\sigma_2\in S_n$ be permutations such that
 $spec(\sigma_1)=spec(\sigma_2)$.
 Then
 $spec(\sigma_1')=spec(\sigma_2')$.
In other words, permutations with the same cycle index  in $S_n$
induce the permutations with the same cycle index in $S_n'$.
\item
The permutation $T_{\rho,\sigma}$ has at least one fixed point if
and only if  $spec(\sigma)=spec(\rho)$.
\item
Let $\sigma \in S_{n,p}$ and let
$spec(\sigma')=p'=(p'_1,p'_2,\ldots,p'_N)$. The number of fixed
points of $T_{\rho, \sigma}$ is
$$N_p=
\prod_{i} i^{p'_i} p'_i!.
$$
\item
If $\sigma\in S_n$ is a cyclic permutation
(a permutation having only one cycle of the length $k>1$), then the cycle index 
monomial of the permutation $\sigma'$ is
$$
\prod_{d|k}f_d^{e(d)},
$$
where the numbers $e(k),k\ge 1$ are defined by the recurrent
relation
$$
e(k)=\frac{1}{k} \left(2^k-\sum_{d|k, d<k}d\cdot e(d)\right), \quad
k>1.
$$
with the initial value $e(1)=2$.
\item
If $\alpha$ is a permutation on a set $X$ with $|X|=a$ and $\alpha$
has a cycle index monomial $f_1^{j_1}\cdots f_a^{j_a}$, and $\beta$
is a permutation on $Y$ with $|Y|=b$ and $\beta$ has a cycle index monomial
$f_1^{k_1}\cdots f_b^{k_b}$, then the permutation $(\alpha,\beta)$
acting on $X\times Y$ by the rule
$$
(\alpha,\beta)(x,y)=(\alpha (x),\beta (y))
$$
has cycle index monomial given by
\begin{equation}\nonumber
\Big(\displaystyle\prod_{p=1}^a f_p^{j_p}\Big)\varprod
\Big(\displaystyle\prod_{q=1}^b
f_q^{k_q}\Big)=\displaystyle\prod_{p=1}^a \displaystyle\prod_{q=1}^b
\big(f_p^{j_p}\times f_q^{ k_q}\big)=\displaystyle\prod_{p=1}^a
\displaystyle\prod_{q=1}^b  f_{<p,q>}^{j_pk_q(p,q)}.
\end{equation}
\end{enumerate}

\section{The number of  equivalence classes}

The value of $V_n$ is determined by the following theorem.

\begin{thm}
For an arbitrary $p\in P_n$ let $\sigma\in S_{n,p}$. If
$spec(\sigma')=(p'_1,\dots,p'_n)$, then
\begin{equation}\label{glavna}
V_n=\sum_{p\in P_n}\frac{\prod_{i} i^{p'_i} p'_i!}{\Big(\prod_{i}
i^{p_i} p_i!\Big)^2}.
\end{equation}
\end{thm}
\begin{proof}
The permutation $F\in S_N$ is a fixed point of $T_{\rho,\sigma}$
 if $T_{\rho,\sigma}(F(X))=F(X)$
 holds for all $X\in B_n$.
 Let  $I(\rho,\sigma)$ be a number of fixed points of $T_{\rho,\sigma}$.
By Frobenius lemma (see e.g. \cite{1}) the number of equivalence
classes is equal to
\begin{equation}\nonumber
V_n=\frac{1}{(n!)^2}\sum_{\sigma \in S_n}\sum_{\rho \in
S_n}I(\rho,\sigma) =\frac{1}{(n!)^2}\sum_{p\in P_n}\sum_{\rho\in
S_{n,p}}\sum_{q\in P_n}\sum_{\sigma\in S_{n,q}}I(\rho,\sigma).
\end{equation}
By the facts (2),(3),(4) from Preliminaries, the number of fixed
points of $T_{\rho,\sigma}$ corresponding to fixed permutations
$\rho\in S_{n,p}$, $\sigma\in S_{n,q}$ is equal to
$$
I(\rho,\sigma)=\left\{
\begin{array}{ll}
0,\ p\neq q\\
N_p,\ p=q
\end{array}
\right.
$$
Therefore
\begin{eqnarray*}
V_n &=&
       \frac{1}{(n!)^2}\sum_{p\in P_n}\sum_{\rho\in S_{n,p}}\sum_{q\in \{p\}}\sum_{\sigma\in S_{n,p}}N_p
       = \frac{1}{(n!)^2}\sum_{p\in P_n}\sum_{\rho\in S_{n,p}}\sum_{\sigma\in S_{n,p}}N_p \\
      & =& \frac{1}{(n!)^2}\sum_{p\in P_n}N_p\sum_{\rho\in S_{n,p}}\sum_{\sigma\in S_{n,p}}1
     = \frac{1}{(n!)^2}\sum_{p\in P_n}N_p\cdot |S_{n,p}|^2\\
      & =& \sum_{p\in P_n}\frac{\prod_{i} i^{p'_i}
p'_i!}{\Big(\prod_{i} i^{p_i} p_i!\Big)^2}
\end{eqnarray*}
\end{proof}

By induction the following generalization of the fact (6) can be
proved. If $\alpha_i$ is permutation on $Z_i$, $|Z_i|=k_{i},
i=1,\dots,n$, and if the cycle index monomial of $\alpha_i$ is
$f_1^{y_{i,1}}\cdots f_{k_i}^{y_{i,k_i}}$, then the permutation 
$(\alpha_1,\dots,\alpha_n)$ acting on  $Z_1\times Z_2\times \cdots
\times Z_n$ by the rule:
$$
(\alpha_1,\dots,\alpha_n)(z_1,\dots,z_n)=(\alpha_1(z_1),\dots,\alpha_n(z_n))
$$
has cycle index monomial given by:
\begin{equation}\label{velika}
\begin{split}
\varprod_{i=1}^n \Big(\prod_{z_i=1}^{k_{i}}
f_{z_i}^{y_{i,z_i}}\Big) & =
\prod_{z_1=1}^{k_1}\prod_{z_2=1}^{k_2}\cdots\prod_{z_n=1}^{k_n}
\varprod_{i=1}^n f_{z_i}^{y_{i,z_i}}\\ & = 
\prod_{z_1=1}^{k_1}\prod_{z_2=1}^{k_2}
\cdots\prod_{z_n=1}^{k_n}f_{<z_1,z_2,\dots,z_n>}^{\prod_{i=1}^n (z_i
y_{i,z_i})/<z_1,z_2,\dots,z_n>}
\end{split}
\end{equation}
The proof is based on the fact, also proved by induction, that the cycle index monomial of the 
direct product of $n$ permutations with cycle index monomials
$f_{z_i}^{y_i},\ 1\le i\le n$ is equal to
$$
\varprod_{i=1}^n
f_{z_i}^{y_i}=f_{<z_1,z_2,\dots,z_n>}^{\prod_{i=1}^n
(z_iy_i)/<z_1,z_2,\dots,z_n>}
$$

Using this generalization the following theorem shows how to obtain
cycle index $p'$ of $\sigma'$,  used in previous theorem. 
\begin{thm}
Let $p\in P_n$ be an arbitrary partition and let $\sigma\in
S_{n,p}$. Let $\sigma=\alpha_1\alpha_2\ldots\alpha_{m}$ be a
decomposition of $\sigma$ into disjoint cycles. Let the length of
$\alpha_i$ be $k_i$, $1\le i\le m$. The cycle index monomial
$\prod_{i} f_i^{p_i'}$ of the corresponding $\sigma'$ is given by
$$
\varprod_{i=1}^m \Big(\prod_{z_i|k_{i}} f_{z_i}^{e(z_i)}\Big)= 
\prod_{z_1|k_{1}}\prod_{z_2|k_{2}}
\cdots\prod_{z_{m}|k_{m}}
f_{<z_1,z_2,\dots,z_{m}>}^{\prod_{i=1}^{m}z_ie(z_i)/<z_1,z_2,\dots,z_{m}>}
\equiv \prod_{i} f_i^{p_i'}.
$$
\end{thm}

\begin{proof}
The cycle of length $k_i$ in $\sigma$ induces the product of cycles in $\sigma'$ 
with the cycle index monomial $\prod_{z_i|k_i} f_{z_i}^{e(z_i)}$.
The product of permutations with cycle index monomial $\prod_{i=1}^{n}t_i^{p_i}=\prod_{i=1}^m t_{k_i}$ in 
$\sigma$ induces a permutation with the cycle index monomial 
$\varprod_{i=1}^m\prod_{z_i|k_{i}}f_{z_i}^{e(z_i)}$ in $\sigma'$.
The cycle index of  $\sigma'$ is then obtained using \eqref{velika}
$$
\prod_{i} f_i^{p_i'}= \prod_{z_1|k_{1}}\prod_{z_2|k_{2}}
\cdots\prod_{z_{m}|k_{m}} 
f_{<z_1,z_2,\dots,z_{m}>}^{\prod_{i=1}^{m}z_ie(z_i)/<z_1,z_2,\dots,z_{m}>}.
$$
\end{proof}

The values $V_n$, $n\le 9$, are listed in the Appendix 1. The
Mathematica code used to compute $V_n$, $n\le 30$,
is given in Appendix 2. 

The following diagram displays the dependence of the computation
time on $n$.
More precisely, the natural logarithms of the two times (in seconds), denoted
by $T_n$ and $T_n'$, respectively, are displayed --- the time needed
to compute $V_n$, and the time needed to compute only cycle indexes
 of $\sigma\in S_{n,p}$ and $\sigma'$ for all partitions
$p\in P_n$. It is seen that the most time-consuming part of the
algorithm is the calculation including large numbers.

\begin{figure}[htb]
\begin{center}
\includegraphics[width=99mm]{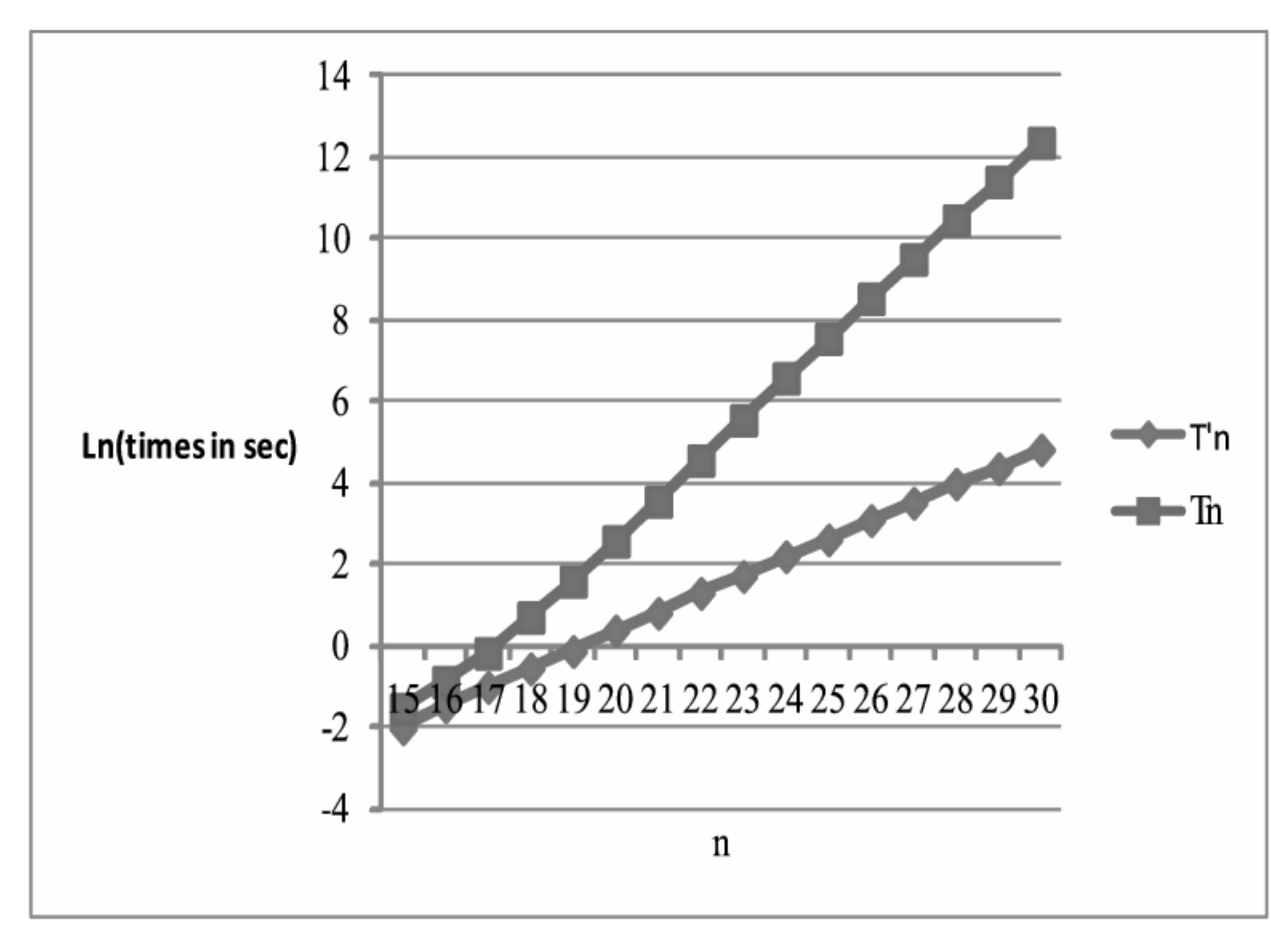}
\caption{Computation time.} \label{labela}
\end{center}
\end{figure}

\section{Acknowledgement}

We are greatly indebted to the anonymous referee for many useful
comments.

\bibliographystyle{amsplain}

\begin{thebibliography}{4}

\bibitem{1} C.\,S. Lorens,
\emph{Invertible Boolean functions}, IEEE Trans. Electron. Comput.
vol. EC-13 (1964), 529-541.

\bibitem{2} M.\,A. Harrison,
\emph{The number of transitivity sets of boolean functions}, J. Soc.
Ind. Appl. Math. Vol. 11, No. 3 \textbf{  } (1963), 806-828.

\bibitem{3} M.\,A. Harrison,
\emph{Counting theorems and their applications to switching theory},
Chapter 4 in A. Mukhopadyay, ed., \emph{Recent Developments in
Switching Functions}, Academic Press, New York, 1971. 85-120.

\bibitem{4}The On-Line Encyclopedia of Integer Sequences, published electronically at http://oeis.org, 2010.

\end{thebibliography}

\newpage
\appendix{Appendix 1}
\\
\begin{center}
  \begin{tabularx}{\textwidth}{| l | X |}
    \hline
    $n$ & $V_n$ \\ \hline
    1 & 2  \\ \hline
    2 & 7  \\ \hline
    3 & 1172  \\ \hline
    4 & 36325278240  \\ \hline
    5 & 18272974787063551687986348306336  \\ \hline
    6 &
244766458691906180755079840538506099505695351680436638205
95072184452539763881615360  \\ \hline
    7 &
151809524739614993439656189343767180263248421197965321191
904605724195878438824958585282876008649113038308707681433
094481099339034671518287392649956419149356667671350869456
39514180251383218789709596490606182400  \\ \hline
    8 &
527659782837770710136789040651740826093251458741455766971
652161168778155554883593317842362390835744875800075919391
263927615040018821216278774202023997397079366824661489931
236149172145693732111500392430562009227866877508698468121
554937012501451700073196915851213949214655962172717654540
696753659200149845589856184612118110722871947349006968239
825032973991355317680321485829624176524898666616495869646
130020407180560252107684411506255014485923878233106586773
392428630175571150605101668259332096000000  \\  \hline
   9 &
264067418605773442099193092831568856287068319670491812908
067855952079195791578203743795004998006943439319356656385
856600567710014711080138640649077160404172738505092500185
193919142842736546598032225971952796396301701808139278439
576519763761893112484051344864646495243044248730968463365
938697655293291537378793983074952990710361271944889122793
368550058901512007397664533157286038662914565247780734211
317088175814665483522871141960413322832647060844552347256
373397308619125110109047501219668866679980778072621664440
855318152480776920392080852178579766887332338861885660254
769972145595031440037588240231880327758293339576810734652
870268328095377818137864938964385302579227441411563577111
620506612132095088626377113807622856129578420710344181832
189204911264827583469568519167634378005415960387723763409
541076277771078470041148992652485171561684276863644504086
228172826519563213828486492418364247504800661684726454045
647372343309018751738150295146135226912754254972781877297
002209445794731239211578043637320353520211584485631583157
724614475661531057325470347588744693853515196047874741937
656673533964435465856256595927257366771432995467331495494
9427200000000000\\ \hline
  \end{tabularx}
\end{center}
\newpage
\appendix{Appendix 2 - Mathematica code

\begin{verbatim}
n = 4; (* the value of n is chosen here *)
e = Table[2, {n}];(*the sequence e*)
Do[
 DD = Divisors[k];
 e[[k]] = (2^k - Sum[DD[[j]] e[[DD[[j]]]], {j, 1, Length[DD] - 1}])/
   k, {k, 2, n}]
PP = IntegerPartitions[n]; npp =
 Length[PP];(*the list of partitions of n*)
(*the maximum length of a cycle in sigma'*)
mlcm = Apply[Max, Table[Apply[LCM, PP[[p]]], {p, npp}]];
(*decompositions of n corresponding to partitions*)
P = Table[0, {i, npp}, {j, n}];
Do[Do[P[[ipp, PP[[ipp, i]]]]++, {i, Length[PP[[ipp]]]}], {ipp, npp}]
EmptyList = Table[0, {j, mlcm}];(*used to initialize spec(sigma')*)
Vn = 0; Do[(*the main loop through all partitions of n*)
 PPP = PP[[p]]; np = Length[PPP];(*current partition*)
 Spec = EmptyList;(*initialization of spec(sigma')*)
 divsets = {};
 nd = 1;
 Do[(*k is the index of the current Partition element*)
  DD = Divisors[PPP[[k]]];
  AppendTo[divsets, DD];
  nd *= Length[DD], {k, 1, np}];
 (*divsets is the list of the sets of divisors of cycle lengths in \
sigma*)
 Descartes = Tuples[divsets]; (* nd is the length of Descartes *)
 Do[ (*loop through Descartes product *)
  product = Descartes[[id]];
  npr = Length[product];
  lcm = 1; prx = 1; pry = 1;
  (* Theorem 2 *)
  Do[
   lcm = LCM[lcm, product[[ipr]]];
   prx *= product[[ipr]];
   pry *= e[[product[[ipr]]]], {ipr, npr}];
  Spec[[lcm]] += prx*pry/lcm, {id, nd}];
 (* Theorem 1 *)
 numerator = Product[i^Spec[[i]] Spec[[i]]!, {i, Length[Spec]}];
 denominatorr = Product[i^P[[p, i]] P[[p, i]]!, {i, n}];
 sum = numerator/denominatorr^2;
 Vn += sum, {p, npp}]
Print[{"V_n = ", Vn}]
\end{verbatim}

\end{document}

@author: @affiliation: @title: @language: English @pages:
@classification1: @classification2: @keywords: @abstract: @filename:
@EOI